\documentclass[10pt, a4paper]{amsart}
\usepackage{amssymb, amsmath, amsfonts, amsthm, verbatim}
\usepackage{hyperref}
\newtheorem{theo}{Theorem}[section]
\newtheorem{lemma}[theo]{Lemma}
\newtheorem{definition}[theo]{Definition}
\newtheorem{cor}[theo]{Corollary}

\newtheorem{conj}[theo]{Conjecture}
\newtheorem{prop}[theo]{Proposition}
\newtheorem{q}[theo]{Question}

\title{Coxeter groups as Beauville groups}

\author{Ben Fairbairn}
\address{Ben Fairbairn, Department of Economics, Mathematics and Statistics, Birkbeck, University of London, Malet Street, London WC1E 7HX, United Kingdom}
\email{b.fairbairn@bbk.ac.uk}

\begin{document}
\maketitle

\begin{abstract}
We generalize earlier work of Fuertes and Gonz\'{a}lez-Diez as well as earlier work of Bauer, Catanese and Grunewald by classifying which of the irreducible Coxeter groups are (strongly real) Beauville groups. We also make partial progress on the much more difficult question of which Coxeter groups are Beauville groups in general as well as discussing the related question of which Coxeter groups can be used in the construction of mixed Beauville groups.
\end{abstract}



\textbf{Acknowledgement} The author wishes to express his deepest gratitude to his colleague Professor Sarah Hart and to Professor Christopher Parker for extremely helpful comments on earlier drafts of this work.

\section{Introduction}\label{Sec1}
We begin with our main definition.

\begin{definition}\label{MainDef}
Let $G$ be a finite group and for $g,h\in G$ let
\[
\Sigma(g,h):=\bigcup_{i=1}^{|G|}\bigcup_{k\in G}\{(g^i)^k,(h^i)^k,((gh)^i)^k\}.
\]

A set of elements $\{\{x_1,y_1\},\{x_2,y_2\}\}\subset G\times G$ is an
\emph{unmixed Beauville structure of} $G$ if and only if $\langle
x_1,y_1\rangle=\langle x_2,y_2\rangle=G$ and
\begin{equation}
\Sigma(x_1,y_1)\cap\Sigma(x_2,y_2)=\{e\}.\tag{$\dagger$}
\end{equation}

If $G$ has an unmixed
Beauville structure then we call $G$ an \emph{unmixed Beauville
group} or simply a \emph{Beauville group}.
\end{definition}

In some places in the literature authors have stated the above definition in terms of \emph{spherical systems of generators} of length $3$, meaning a triple (as opposed to a pair) of generators $\{x,y,z\}$ with $xyz=e$, but we omit $z=(xy)^{-1}$ from our notation here. Furthermore, many earlier papers on Beauville structures add the condition that for $i=1,2$ we have
\[
o(x_i)^{-1}+o(y_i)^{-1}+o(x_iy_i)^{-1}<1,
\]
but this condition was subsequently found to be unnecessary \cite{BCG1}.

Beauville groups were originally introduced in connection with a class of complex surfaces of general type known as Beauville surfaces. Very roughly speaking these surfaces are defined by taking a product of two Riemann surfaces $\mathcal{C}\times\mathcal{C}'$ and quotienting out the action of a finite group $G$ on this product giving us the variety $(\mathcal{C}\times\mathcal{C}')/G$. These surfaces possess many useful geometric properties: their automorphism groups \cite{Jonesauts} and fundamental groups \cite{Catnese} are relatively easy to compute and these surfaces are rigid surfaces in the sense of admitting no non-trivial deformations \cite{BCG2} and thus correspond to isolated points in the moduli space of surfaces of general type. As a consequence they are useful for testing conjectures and providing cheap counterexamples (most notably the Friedman-Morgan ``speculation" --- see \cite[Section 7]{BCG2} for details). Several authors have produced excellent surveys on these and closely related matters in recent years including \cite{BCPSurvey,FSurvey,JonesSurvey,S,WolfartSurvey}.

Complex surfaces sometimes have an additional structure analogous to complex conjugation. In the case of Beauville surfaces, the existence of such a function can be encoded in terms of the Beauville structure as follows.

\begin{definition}
 Let $G$ be a Beauville group and let $X =\{\{x_1, y_1\},\{x_2, y_2\}\}$ be a
Beauville structure for $G$. We say that $G$ and $X$ are \emph{strongly real} if there exists an
automorphism $\phi\in\mbox{Aut}(G)$ and elements $g_i\in G$ for $i = 1, 2$
such that
\[
g_i\phi(x_i)g_i^{-1}=x_i^{-1}\mbox{ and }g_i\phi(y_i)g_i^{-1}=y_i^{-1}
\]
for $i=1,2$.
\end{definition}

In practice we often take $g_1=g_2=e$ in this definition.

In \cite[Section 3.2]{BCG1} Bauer, Catanese and Grunewald pose the general question ``which [finite] groups admit Beauville structures?" as well as the more difficult question of which Beauville groups are strongly real Beauville groups. Many authors have investigated these questions for several classes of finite groups including abelian groups \cite{Catnese}; symmetric groups \cite{BCG1,FG} as well as decorations of simple groups more generally \cite{FExcep,FJ,FMP1,FMP2,GLL,GM}; characteristically simple groups \cite{MoreF,NewcProc,Jones1,Jones2} and nilpotent groups \cite{BBF,BBPV1,BBPV2,MoreF,new,SV} (this list of references is by no means exhaustive!) Here we extend this list by considering Coxeter groups. First we classify which of the irreducible Coxeter groups are (strongly real) Beauville groups.

\begin{theo}\label{MainThm}
Every finite irreducible Coxeter group is a strongly real Beauville group aside from the groups of type:
\begin{enumerate}
\item[(a)] \emph{A}$_n$ for $n\leq3$;
\item[(b)] \emph{B}$_n$ for $n\leq4$;
\item[(c)] \emph{D}$_n$ for $n\leq4$;
\item[(d)] \emph{F}$_4$ and \emph{H}$_3$ and
\item[(e)] \emph{I}$_2(k)$ for any $k$.
\end{enumerate}
\end{theo}

Whilst the classification of strongly real Beauville Coxeter groups is in general somewhat more tricky (see the comment in Section \ref{proofs}), we can at least get most of the way.

\begin{cor}\label{MainCor}
No product of three or more irreducible Coxeter groups is a Beauville group. Furthermore, $K_1\times K_2$ is a strongly real Beauville group if $K_1$ and $K_2$ are strongly real irreducible Coxeter Beauville groups not of type B$_n$.
\end{cor}

As a consequence of the proof of Theorem \ref{MainThm} we have the following.

\begin{cor}
An irreducible Coxeter group is a Beauville group if and only if it is a strongly real Beauville group.
\end{cor}

For the basic definitions and notation for Coxeter groups used throughout this paper we refer the reader to \cite{Hump}.

This paper is organised as follows. In Sections 2 and 3 we will prove Theorem \ref{MainThm} in the cases of types B$_n$ and D$_n$ respectively before going onto to discuss the exceptional groups in Section 4. In Section 5 we pull together the work of the earlier parts to complete the proof of Theorem \ref{MainThm} and its corollaries along with a discussion as to why it is difficult to extend these results further. Finally in Section 6 we discuss mixed and mixable Beauville groups, in which the curves $\mathcal{C}$ and $\mathcal{C}'$ are interchanged by the action of the group on $\mathcal{C}\times\mathcal{C}'$, that we will define more precisely there.

\section{Coxeter groups of type B$_n$}\label{BnSec}

In this section we prove Theorem \ref{MainThm} for the groups of type B$_n$.

Recall that matrix is `monomial' if it has precisely one non-zero entry in each row and column. For any positive integer $n$ the Coxeter group $W($B$_n)$ has a degree $n$ irreducible ordinary representation in which every element of the group is represented by a monomial matrix in which all of the non-zero entries are equal to $\pm1$. In particular, this means that elements of the group can be represented by expressing a matrix as a permutation with an indication of which columns contain a $-1$ rather than a $1$. We do this by underlining any number corresponding to a column containing a $-1$. This can sometimes lead to a slightly odd way of expressing a permutation with cycles of length one being included to indicate an entry of the matrix equal to $-1$ whilst not including other cycles of length one. For example, we write
\[
(1,\underline{2},3)(\underline{5})\mbox{ for }\left(\begin{array}{rrrrr}
\cdot&\cdot&1&\cdot&\cdot\\
1&\cdot&\cdot&\cdot&\cdot\\
\cdot&-1&\cdot&\cdot&\cdot\\
\cdot&\cdot&\cdot&1&\cdot\\
\cdot&\cdot&\cdot&\cdot&-1
\end{array}\right).
\]

In this section and the next we shall repeatedly make use of the following recent general result of Jones \cite[Proposition 6.1]{Jones1} (the special case in which the length of the cycle must be a prime number is a clasical result due to Jordan, itself generalised by a more recent result of Neumann in which the length of the cycle can be a prime power).

\begin{lemma}\label{JonesLem}
Let $H$ be a primitive permutation group of degree $n$. If $H$ contains a cycle of length $m$ where $2\leq m\leq n-3$, then $H\geq Alt(n)$.
\end{lemma}

We first deal with the case in which $n$ is large and even.

\begin{prop}\label{BnEvenLem}
If $n\geq12$ is even, then $W($\emph{B}$_n)$ is a strongly real Beauville group.
\end{prop}

\begin{proof}
Consider the following elements.
\[
x_1:=(1,2,\ldots,n-1)(\underline{n})\mbox{, }y_1:=(n,n-1,\ldots,4,3)
\]
\[
t_1:=(1,2)(3,n-1)(4,n-2)\cdots(n/2,n/2+2)
\]
\[
x_2:=(1,2,\ldots,n-3)(\underline{n-2})(\underline{n-1})(\underline{n})\mbox{, }y_2:=(n,n-1,\ldots,6,5)
\]
\[
t_2:=(1,4)(2,3)(n,n-2)(5,n-3)(6,n-4)\cdots(n/2,n/2+2)
\]
We claim that $\{\{x_1,y_1\},\{x_2,y_2\}\}$ is a strongly real Beauville structure for $W($B$_n)$.

We first claim that $\langle x_1, y_1\rangle=W($B$_n)$. Viewing the subgroup $H:=\langle x_1^2, y_1\rangle$ as a degree $n$ permutation group it is easy to see that $H$ is 2-transitive and thus primitive. It is easily checked by direct calculation that $x_1^2y_1^2=(1,n-1,2,n,n-2)$ which is a 5-cycle and so by Lemma \ref{JonesLem} $H\geq Alt(n)$. Since $y_1$ is easily seen to be an odd permutation, we must have that $H=Sym(n)$. Since $x_1^{n-1}$ is a diagonal matrix with a sign change in one place, it follows that $\langle x_1, y_1\rangle=W($B$_n)$.

We next claim that $\langle x_2, y_2\rangle=W($B$_n)$. Viewing the subgroup $H:=\langle x_1^2, y_1\rangle$ as a degree $n$ permutation group we note that the subgroup $\langle x_2^2,(x_2^2)^{y_2},(x_2^2)^{y_2^2}\rangle$ fixes $n-2$ and acts transitively on the remaining points, so $H$ is 2-transitive and thus primitive. Since $x_2^2$ is an $(n-3)$-cycle it follows that $H\geq Alt(n)$ by Lemma \ref{JonesLem}. Since $y_2$ is easily seen to be an odd permutation, we must have that $H=Sym(n)$. Since $x_1^{n-1}$ is a diagonal matrix with a sign change in three places, it follows that $\langle x_2, y_2\rangle=W($B$_n)$.

Next we show that condition ($\dagger$) of  Definition \ref{MainDef} is satisfied. To do this we use the well known fact from linear algebra that any two conjugate matrices will have equal traces. For the elements given above these are easily obtained by direct calculation and are listed in Table \ref{EvenBnTrace}. By hypothesis $n\geq12$ and so these numbers are all distinct with the following exceptions. If $n=12$, then $Tr(y_2)=Tr(x_2y_2)$ but this does not obstruct condition ($\dagger$) holding. Furthermore $Tr((x_1y_1)^4)=Tr((x_2y_2)^r)$ when $r\not=8$ however in a monomial representation diagonal matrices can only be conjugate to other diagonal matrices and so $(x_1y_1)^4$ (which is diagonal) and $(x_2y_2)^r$ when $r\not=8$ (which is not diagonal) cannot be conjugate. If $n=18$, then we also have that $Tr(y_1)=Tr((x_2y_2)^8)$ but again only one these two elements is diagonal so they are again not conjugate. We have thus shown that $\{\{x_1,y_1\},\{x_2,y_2\}\}$ is a Beauville structure for $W($B$_n)$.

Finally we show that this Beauville structure is strongly real; this is simply a direct calculation verifying for $i=1,2$ that we have $x_i^{t_i}=x_i^{-1}$ and $y_i^{t_i}=y_i^{-1}$ and visibly conjugation by $t_1$ and $t_2$ differ only by an inner automorphism.
\end{proof}

\begin{table}
\begin{center}
\begin{tabular}{|c|c|}
\hline
$g$&$Tr(g^r)$\\
\hline\hline
$x_1$&$1$ ($r$ even) $-1$ ($r\not=n-1$ odd) $n-2$ ($r=n-1$)\\
$y_1$&$2$\\
$x_1y_1$&$n-4$ ($r\not=4$) $n-8$ ($r=4$)\\
\hline\hline
$x_2$&$3$ ($r$ even) $-3$ ($r\not=n-3$ odd) $n-6$ ($r=n-3$)\\
$y_2$&4\\
$x_2y_2$&$n-8$ ($r\not=8$) $n-16$ ($r=8$)\\

\hline
\end{tabular}
\end{center}
\caption{The traces of non-trivial powers of the elements appearing in the proof of Lemma \ref{BnEvenLem}.}
\label{EvenBnTrace}
\end{table}

Note that the hypothesis in the above lemma that $n\geq12$ cannot be weakened without using a different construction from the one considered in the above proof. Next, $n$ large and odd.

\begin{prop}\label{BnOddLem}
If $n\geq11$ is odd, then $W($\emph{B}$_n)$ is a strongly real Beauville group.
\end{prop}

\begin{proof}
Consider the following elements.
\[
x_1:=(1,2,\ldots,n-2)(\underline{n-1})\mbox{, }y_1:=(n,n-1,\ldots,2)
\]
\[
t_1:=(2,n-2)(3,n-3)\cdots((n-1)/2,(n+1)/2)(n-1,n)
\]
\[
x_2:=(1,2,\ldots,n-4)(\underline{n-3})(\underline{n-2})(\underline{n-1})\mbox{, }y_1:=(n,n-1,\ldots,4)
\]
\[
t_2:=(1,3)(4,n-4)\cdots((n-1)/2,(n+1)/2)
\]
We claim that $\{\{x_1,y_1\},\{x_2,y_2\}\}$ is a strongly real Beauville structure.

We first claim that $\langle x_1, y_1\rangle=W($B$_n)$. Viewing the subgroup $H:=\langle x_1^2, y_1\rangle$ as a degree $n$ permutation group, the subgroup $\langle x_1^2, (x_1)^{y_1}\rangle$ fixes $n-1$ and acts transitively on the remaining points, so $H$ is 2-transitive and thus primitive. It is easily checked by direct calculation that $x_1^2y_1^2=(1,n,n-2,n-1,n-3)$ which is a 5-cycle and so by Lemma \ref{JonesLem} $H\geq Alt(n)$. Since $y_1$ is easily seen to be an odd permutation, we must have that $H=Sym(n)$. Since $x_1^{n-2}$ is a diagonal matrix with a sign change in one place, it follows that $\langle x_1, y_1\rangle=W($B$_n)$.

We next claim that $\langle x_2, y_2\rangle=W($B$_n)$. Viewing the subgroup $H:=\langle x_1^2, y_1\rangle$ as a degree $n$ permutation group we note that the subgroup $\langle x_2^2,(x_2^2)^{y_2},(x_2^2)^{y_2^2},(x_2^2)^{y_2^3}\rangle$ fixes $n-3$ and acts transitively on the remaining points, so $H$ is 2-transitive and thus primitive. Since $y_2$ is an $(n-3)$-cycle it follows that $H\geq Alt(n)$ by Lemma \ref{JonesLem}. Since $y_2$ is easily seen to be an odd permutation, we must have that $H=Sym(n)$. Since $x_1^{n-1}$ is a diagonal matrix with a sign change in three places, it follows that $\langle x_2, y_2\rangle=W($B$_n)$.

Next we show that condition ($\dagger$) of Definition \ref{MainDef} is satisfied. To do this we again use the well known fact from linear algebra that any two conjugate matrices will have equal traces. For the elements given above these are easily obtained by direct calculation and are listed in Table \ref{OddBnTrace}. By hypothesis $n\geq11$ and so these numbers are all distinct with the following exceptions. If $n=11$, then  $Tr((x_2y_2)^l)=Tr(y_2^m)$ for $l\not=8$ and $m\not=o(y_2)$ and if $n=19$, then $Tr((x_2y_2)^8)=Tr(y_2^r)$ but neither of these calculations obstruct condition ($\dagger$) holding.

Furthermore, if $n=11$, then $Tr((x_1y_1)^4)=Tr(y_2^r)$ for any $r\not=o(y_2)$ and if $n=17$, then $Tr((x_2y_2)^8)=Tr(y_1^r)$ for any $r\not=o(y_1)$ however in a monomial representation diagonal matrices can only be conjugate to other diagonal matrices and so neither $(x_1y_1)^4$ nor $(x_2y_2)^8$ (which are diagonal) can be conjugate to $y_1^l$ for $l\not=o(y_1)$ or $y_2^m$ for $m\not=o(y_2)$ (which are not diagonal). We have thus shown that $\{\{x_1,y_1\},\{x_2,y_2\}\}$ is a Beauville structure for $W($B$_n)$.

Finally we show that this Beauville structure is strongly real; this is simply a direct calculation verifying that for $i=1,2$ we have that $x_i^{t_i}=x_i^{-1}$ and $y_i^{t_i}=y_i^{-1}$ and visibly conjugation by $t_1$ and $t_2$ differ only by an inner automorphism.
\end{proof}

\begin{table}
\begin{center}
\begin{tabular}{|c|c|}
\hline
$g$&$Tr(g^r)$\\
\hline\hline
$x_1$&$2$ ($r$ even) $0$ ($r\not=n-2$ odd) $n-2$ ($r=n-2$)\\
$y_1$&$1$\\
$x_1y_1$&$n-4$ ($r\not=4$) $n-8$ ($r=4$)\\

\hline\hline

$x_2$&$4$ ($r$ even) $-2$ ($r\not=n-4$ odd) $n-6$ ($r=n-4$)\\
$y_2$&3\\
$x_2y_2$&$n-8$ ($r\not=8$) $n-16$ ($r=8$)\\

\hline
\end{tabular}
\end{center}
\caption{The traces of non-trivial powers of the elements appearing in the proof of Lemma \ref{BnOddLem}.}
\label{OddBnTrace}
\end{table}

Whilst some would simply leave the small cases to the reader we prefer to be more explicit, given the history of the Alt(6) controversy: \cite{FG} states that ``On the other hand computer inspection using the GAP programme has revealed that no strongly real structure exists on [Alt(6)]" whilst \cite{FJ} claimed the opposite, again without being explicit. (It turns out that Alt(6) is indeed a strongly real Beauville group.)

\begin{lemma}\label{BnSmallLem}
The Coxeter groups of type \emph{B}$_n$ for $5\leq n\leq10$ are strongly real Beauville groups.
\end{lemma}

\begin{proof}
A straightforward computation verifies that the elements of these groups given in Table \ref{BnSmall} provide strongly real Beauville structures.
\end{proof}

\begin{table}
\begin{center}
\begin{tabular}{|c|l|l|}
\hline

$n$&$x_1$&$x_2$\\
&$y_1$&$y_2$\\
&$t_1$&$t_2$\\
\hline\hline
5&$(\underline{1},\underline{2},\underline{5})$&$(2,3,4,5)$\\
&(1,2,3)(4,5)&(1,4,\underline{2},5,3)\\
&(1,2)&(2,5)(3,4)\\

6&(1,2,3,4,5)(\underline{6})&(1,2,3,4)(\underline{5},6)\\
&(6,5,4,3,2,1)&(6,5,4,3)\\
&(1,5)(2,4)&(1,2)(3,4)(5,6)\\

7&$(\underline{1})(2,6)(4,7)$&$(1,2,3,4,5)(\underline{6})$\\
&(1,2,3,4)(5,6,7)&(1,7)(6,5,4,3,2)\\
&(2,4)(6,7)&(2,5)(3,4)\\

8&$(1,2,3,4,5,6,7)(\underline{8})$&$(1,2,3,4,5)(\underline{6})(\underline{7})(\underline{8})$\\
&(8,7,6,5,4,3)&(1,3)(8,7,6,5,4)\\
&(1,2)(3,7)(4,6)&(1,3)(4,5)(6,8)\\

9&$(1,2,3,4,5,6,7)(\underline{8})$&$(1,2,3,4,5,6)(\underline{7})(\underline{8})(\underline{9})$\\
&(9,7,6,5,8,3,2,1)&(2,7)(6,8)(1,9)\\
&(1,7)(2,6)(3,5)&(2,6)(5,3)(7,8)\\

10&(1,2,3,4,5,6,7,8,9)(\underline{10})&(1,2,3,4,5,6,7)(\underline{8})(\underline{9})(\underline{10})\\
&(10,9,8,7,6,5,4,3,2,1)&(10,9,8,7,6,5,4,3)\\
&(1,9)(2,8)(3,7)(4,6)&(1,2)(3,7)(4,6)(8,10)\\

\hline
\end{tabular}
\end{center}
\caption{Elements of the groups B$_n$ for small $n$ that provide strongly real Beauville structures proving Lemma \ref{BnSmallLem}.}
\label{BnSmall}
\end{table}

Straightforward computation verifies that $W($B$_n)$ is not a Beauville group for $n\leq4$.

\section{Coxeter groups of type D$_n$}\label{DnSec}

In this section we prove Theorem \ref{MainThm} for the groups of type D$_n$. Since many of the proofs appearing here are analogous to those used in the previous section we merely sketch the details here.

In terms of the representation described at the beginning of Section \ref{BnSec}, the Coxeter group of type D$_n$ is the subgroup of the Coxeter group of type B$_n$ in which every element is represented by a matrix that contains an even number of $-1$s. It follows that we can use the same representation to describe elements of these groups.

\begin{prop}\label{DnEvenLem}
If $n\geq10$ is even, then $W($\emph{D}$_n)$ is a strongly real Beauville group.
\end{prop}

\begin{proof}
\[
x_1:=(1,2,\ldots,n-3)(\underline{n-2})(\underline{n})\mbox{, }
y_1:=(n,n-1,\ldots,1)
\]
\[
t_1:=(1,n-3)(2,n-4)\cdots(n/2-2,n/2)(n-2,n)
\]
\[
x_2:=(1,2,\ldots,n-5)(\underline{n-4})(\underline{n-3})(\underline{n-1})(\underline{n})\mbox{, }
y_2:=(n,n-1,\ldots,3)
\]
\[
t2:=(1,2)(3,n-5)(4,n-6)\cdots(n/2,n/2-2)(n-4,n)(n-3,n-1)
\]
A proof analogous to that of Lemma \ref{BnEvenLem} may be used to show that $\{\{x_1,y_1\},\{x_2,y_2\}\}$ is a strongly real Beauville structure. The traces of powers of these elements are explicitly listed in Table \ref{EvenDnTrace}.
\end{proof}
\begin{table}
\begin{center}
\begin{tabular}{|c|c|}
\hline
$g$&$Tr(g^r)$\\
\hline\hline
$x_1$&$3$ ($r$ even) $-1$ ($r\not=n-3$ odd) $n-4$ ($r=n-3$)\\
$y_1$&$0$\\
$x_1y_1$&$n-4$\\

\hline\hline

$x_2$&$5$ ($r$ even) $-3$ ($r\not=n-5$ odd) $n-8$ ($r=n-5$)\\
$y_2$&2\\
$x_2y_2$&$n-8$\\

\hline
\end{tabular}
\end{center}
\caption{The traces of non-trivial powers of the elements appearing in the proof of Lemma \ref{DnEvenLem}.}
\label{EvenDnTrace}
\end{table}

\begin{prop}\label{DnOddLem}
If $n\geq11$ is odd, then $W($\emph{D}$_n)$ is a strongly real Beauville group.
\end{prop}

\begin{proof}
Consider the following elements.
\[
x_1:=(1,2,\ldots,n-2)(\underline{n-1})(\underline{n})\mbox{, }
y_1:=(n,n-1,\ldots,4,3,2)
\]
\[
t_1:=(2,n-2)(3,n-3)\cdots((n-1)/2,(n+1)/2)(n-1,n)
\]
\[
x_2:=(1,2,\ldots,n-4)(\underline{n-3})(\underline{n-2})(\underline{n-1})(\underline{n})\mbox{, }
y_2:=(n,n-1,\ldots,4)
\]
\[
t_2:=(1,3)(4,n-4)(5,n-5)\cdots((n-1)/2,(n+1)/2)(n-4,n)(n-3,n-2)
\]
A proof analogous to that of Lemma \ref{BnOddLem} may be used to show that $\{\{x_1,y_1\},\{x_2,y_2\}\}$ is a strongly real Beauville structure. The traces of powers of these elements are explicitly listed in Table \ref{OddDnTrace}.
\end{proof}
\begin{table}
\begin{center}
\begin{tabular}{|c|c|}
\hline
$g$&$Tr(g^r)$\\
\hline\hline
$x_1$&$2$ ($r$ even) $-2$ ($r\not=n-2$ odd) $n-4$ ($r=n-2$)\\
$y_1$&$1$\\
$x_1y_1$&$n-4$\\

\hline\hline

$x_2$&$4$ ($r$ even) $-4$ ($r\not=n-4$ odd) $n-8$ ($r=n-4$)\\
$y_2$&3\\
$x_2y_2$&$n-8$\\

\hline
\end{tabular}
\end{center}
\caption{The traces of non-trivial powers of the elements appearing in the proof of Lemma \ref{DnOddLem}.}
\label{OddDnTrace}
\end{table}

\begin{lemma}\label{DnSmallLem}
The Coxeter groups of type D$_n$ for $5\leq n\leq9$ are strongly real Beauville groups.
\end{lemma}

\begin{proof}
A straightforward computation verifies that the elements of these groups given in Table \ref{DnSmall} provide strongly real Beauville structures.
\end{proof}

\begin{table}
\begin{center}
\begin{tabular}{|c|l|l|}
\hline

$n$&$x_1$&$x_2$\\
&$y_1$&$y_2$\\
&$t_1$&$t_2$\\
\hline\hline
5&(1,2,3,4,5)&(2,4,3)\\
&(1,2)(\underline{4})(\underline{5})&(5,2,3,4)\\
&(1,2)(4,5)&(2,4)\\

6&(1,2,3,4,5)&(\underline{1},\underline{2},\underline{3},\underline{4})(5,6)\\
&(\underline{1})(\underline{4})(5,6)&(6,5,4,3)\\
&(1,4)(2,3)&(1,2)(3,4)(5,6)\\

7&(1,2,3,4,5,6,7)&(1,2,3,4,5)\\
&(1,2)(\underline{3})(\underline{4})&(\underline{7},\underline{6},\underline{5},\underline{4})(3,2,1)\\
&(1,2)(3,7)(4,6)&(1,3)(4,5)(6,7)\\

8&(1,2,3,4,5)(\underline{6})(\underline{7})&(1,2,3,4)(\underline{5},\underline{6},\underline{7},\underline{8})\\
&(1,2,3,4,5,6,7,8)&(8,7,6,5,4,3)\\
&(1,5)(2,4)(6,8)&(1,2)(3,4)(5,8)(6,7)\\

9&(1,2,3,4,5,6)(\underline{7})(\underline{9})&(1,2,3,4,5)(\underline{6},\underline{7},\underline{8},\underline{9})\\
&(9,8,7,6,5,4)(1,3)&(9,8,7,6,5)(4,3,2,1)\\
&(1,3)(4,6)(7,9)&(1,4)(2,3)(6,9)(7,8)\\
\hline
\end{tabular}
\end{center}
\caption{Elements of the groups D$_n$ for small $n$ that provide strongly real Beauville structures proving Lemma \ref{DnSmallLem}.}
\label{DnSmall}
\end{table}

Straightforward computation verifies that $W($D$_n)$ is not a Beauville group for $n\leq4$.

\section{The exceptional groups}\label{ExSec}

In this section we prove Theorem \ref{MainThm} in the case of the exceptional Coxeter groups, that is,
the groups of types E$_6$, E$_7$, E$_8$, F$_4$, H$_3$ and
H$_4$.

\subsection{The Coxeter groups of type E$_6$, E$_7$ and E$_8$}

In each of these three cases we exhibit explicit matrices for strongly real Beauville structures for these groups. The matrices we will use in this section come from representations of these groups first described by the author in \cite[Section 6]{Edin}. These have the helpful property that the parabolic subgroup  isomorphic to a symmetric group obtained by removing a single node is represented by permutation matrices. (Many of the elements lying outside this `monomial subgroup' also admit an unusually simple description.)

\begin{lemma}
The groups $W(\emph{E}_n)$ for $n=6,7,8$ are strongly real Beauville groups.
\end{lemma}

\begin{proof}
It is a straightforward computation to verify that the matrices given in Figures \ref{e6fig}, \ref{e7fig} and \ref{e8fig} have the correct properties to be Beauville structures for these groups as well as the fact that $x_i^t=x_i^{-1}$ and $y_i^t=y_i^{-1}$ for $i=1,2$ in each case thereby showing that these are in fact strongly real Beauville structures.
\end{proof}

\begin{figure}

\[
x_1:=\left(\begin{array}{rrrrrr}
\cdot&\cdot&1&\cdot&\cdot&\cdot\\
1&\cdot&\cdot&\cdot&\cdot&\cdot\\
\cdot&\cdot&\cdot&\cdot&1&\cdot\\
\cdot&1&\cdot&\cdot&\cdot&\cdot\\
\cdot&\cdot&\cdot&\cdot&\cdot&1\\
\cdot&\cdot&\cdot&1&\cdot&\cdot\\
\end{array}\right)\mbox{, }
y_1:=\frac{1}{3}\left(\begin{array}{rrrrrr}
1&-2&1&-2&1&1\\
1&1&1&-2&1&-2\\
-1&-1&2&-1&-1&-1\\
\cdot&\cdot&3&\cdot&\cdot&\cdot\\
-2&1&1&-2&1&1\\
1&1&1&-2&-2&1\\

\end{array}\right)
\]

\[
x_2:=\frac{1}{3}\left(\begin{array}{rrrrrr}
-2&-2&1&1&1&1\\
1&-2&-2&1&1&1\\
\cdot&\cdot&\cdot&3&\cdot&\cdot\\
-2&1&-2&1&1&1\\
\cdot&\cdot&\cdot&\cdot&\cdot&3\\
\cdot&\cdot&\cdot&\cdot&3&\cdot
\end{array}\right)\mbox{, }
y_2:=\frac{1}{3}\left(\begin{array}{rrrrrr}
\cdot&3&\cdot&\cdot&\cdot&\cdot\\
-2&1&1&1&1&-2\\
\cdot&\cdot&3&\cdot&\cdot&\cdot\\
\cdot&\cdot&\cdot&3&\cdot&\cdot\\
-2&1&1&1&-2&1\\
1&1&1&1&-2&-2

\end{array}\right)
\]

\[
t:=\left(\begin{array}{rrrrrr}
\cdot&1&\cdot&\cdot&\cdot&\cdot\\
1&\cdot&\cdot&\cdot&\cdot&\cdot\\
\cdot&\cdot&\cdot&1&\cdot&\cdot\\
\cdot&\cdot&1&\cdot&\cdot&\cdot\\
\cdot&\cdot&\cdot&\cdot&\cdot&1\\
\cdot&\cdot&\cdot&\cdot&1&\cdot
\end{array}\right)
\]

\caption{The matrices verifying that $W$(E$_6)$ is a strongly real Beauville group.}
\label{e6fig}
\end{figure}

\begin{figure}

\[
x_1:=\frac{1}{3}\left(\begin{array}{rrrrrrr}
   1&  -2&   1&   1&  -2&   1&   1\\
     \cdot&     \cdot&     3&     \cdot&     \cdot&     \cdot&     \cdot\\
   1&   1&   1&   1&  -2&   1&  -2\\
     3&     \cdot&     \cdot&     \cdot&     \cdot&     \cdot&     \cdot\\
\cdot&     \cdot&     \cdot&     \cdot&     \cdot&     3&     \cdot\\
  1&  -2&   1&   1&   1&   1&  -2\\
\cdot&     \cdot&     \cdot&     3&     \cdot&     \cdot&     \cdot\\
\end{array}\right)
\mbox{, }
y_1:=\frac{1}{3}\left(\begin{array}{rrrrrrr}
1&1&-2&1&1&1&-2\\
1&1&-2&1&1&-2&1\\
\cdot&\cdot&\cdot&\cdot&3&\cdot&\cdot\\
\cdot&\cdot&\cdot&3&\cdot&\cdot&\cdot\\
-1&-1&-1&2&2&-1&-1\\
-2&1&-2&1&1&1&1\\
1&-2&-2&1&1&1&1
\end{array}\right)
  \]

  \[
x_2:=\left(\begin{array}{rrrrrrr}
\cdot&\cdot&\cdot&\cdot&\cdot&\cdot&-1\\
-1&\cdot&\cdot&\cdot&\cdot&\cdot&\cdot\\
\cdot&-1&\cdot&\cdot&\cdot&\cdot&\cdot\\
\cdot&\cdot&-1&\cdot&\cdot&\cdot&\cdot\\
\cdot&\cdot&\cdot&-1&\cdot&\cdot&\cdot\\
\cdot&\cdot&\cdot&\cdot&-1&\cdot&\cdot\\
\cdot&\cdot&\cdot&\cdot&\cdot&-1&\cdot\\
\end{array}\right)
\mbox{, }
y_2:=\frac{1}{3}\left(\begin{array}{rrrrrrr}
\cdot&\cdot&\cdot&\cdot&\cdot&3&\cdot\\
\cdot&\cdot&\cdot&\cdot&\cdot&\cdot&3\\
1&-2&-2&1&1&1&1\\
-2&1&-2&1&1&1&1\\
-1&-1&-1&-1&-1&2&2\\
1&1&-2&1&-2&1&1\\
1&1&-2&-2&1&1&1
\end{array}\right)
  \]

\[
t:=\left(\begin{array}{rrrrrrr}
\cdot&\cdot&\cdot&\cdot&\cdot&\cdot&1\\
\cdot&\cdot&\cdot&\cdot&\cdot&1&\cdot\\
\cdot&\cdot&\cdot&\cdot&1&\cdot&\cdot\\
\cdot&\cdot&\cdot&1&\cdot&\cdot&\cdot\\
\cdot&\cdot&1&\cdot&\cdot&\cdot&\cdot\\
\cdot&1&\cdot&\cdot&\cdot&\cdot&\cdot\\
1&\cdot&\cdot&\cdot&\cdot&\cdot&\cdot\\
\end{array}\right)
\]

\caption{The matrices verifying that $W$(E$_7)$ is a strongly real Beauville group.}
\label{e7fig}
\end{figure}

\begin{figure}

\[
x_1:=\left(\begin{array}{rrrrrrrr}
\cdot&\cdot&\cdot&\cdot&1&\cdot&\cdot&\cdot\\
1&\cdot&\cdot&\cdot&\cdot&\cdot&\cdot&\cdot\\
\cdot&\cdot&\cdot&\cdot&\cdot&\cdot&1&\cdot\\
\cdot&\cdot&\cdot&\cdot&\cdot&\cdot&\cdot&1\\
\cdot&\cdot&\cdot&\cdot&\cdot&1&\cdot&\cdot\\
\cdot&1&\cdot&\cdot&\cdot&\cdot&\cdot&\cdot\\
\cdot&\cdot&\cdot&1&\cdot&\cdot&\cdot&\cdot\\
\cdot&\cdot&1&\cdot&\cdot&\cdot&\cdot&\cdot\\
\end{array}\right)\mbox{, }
y_1:=\frac{1}{3}\left(\begin{array}{rrrrrrrr}
1&1&1&-2&1&1&-2&-2\\
1&1&1&-2&-2&1&-2&1\\
\cdot&\cdot&\cdot&-3&\cdot&\cdot&\cdot&\cdot\\
1&1&1&-2&-2&1&1&-2\\
\cdot&\cdot&\cdot&-3&\cdot&3&\cdot&\cdot\\
-1&2&-1&-1&-1&2&-1&-1\\
-1&-1&2&-1&-1&2&-1&-1\\
2&-1&-1&-1&-1&2&-1&-1
\end{array}\right)
\]

\[
x_2:=\left(\begin{array}{rrrrrrrr}
1&\cdot&\cdot&\cdot&\cdot&\cdot&\cdot&\cdot\\
\cdot&1&\cdot&\cdot&\cdot&\cdot&\cdot&\cdot\\
\cdot&\cdot&\cdot&\cdot&\cdot&1&\cdot&\cdot\\
\cdot&\cdot&1&\cdot&\cdot&\cdot&\cdot&\cdot\\
\cdot&\cdot&\cdot&1&\cdot&\cdot&\cdot&\cdot\\
\cdot&\cdot&\cdot&\cdot&\cdot&\cdot&1&\cdot\\
\cdot&\cdot&\cdot&\cdot&1&\cdot&\cdot&\cdot\\
\cdot&\cdot&\cdot&\cdot&\cdot&\cdot&\cdot&1\\
\end{array}\right)\mbox{, }
y_2:=\frac{1}{3}\left(\begin{array}{rrrrrrrr}
2&2&-1&-1&2&-1&-1&-1\\
2&2&-1&-1&-1&-1&-1&2\\
1&1&1&-2&1&-2&-2&1\\
1&1&-2&1&1&-2&-2&1\\
2&2&-1&-1&-1&-1&-1&-1\\
\cdot&3&\cdot&\cdot&\cdot&-1&\cdot&\cdot\\
1&1&-2&-2&1&-2&1&1\\
3&\cdot&\cdot&\cdot&\cdot&-3&\cdot&\cdot
\end{array}\right)
\]

\[
t:=\left(\begin{array}{rrrrrrrr}
\cdot&1&\cdot&\cdot&\cdot&\cdot&\cdot&\cdot\\
1&\cdot&\cdot&\cdot&\cdot&\cdot&\cdot&\cdot\\
\cdot&\cdot&\cdot&1&\cdot&\cdot&\cdot&\cdot\\
\cdot&\cdot&1&\cdot&\cdot&\cdot&\cdot&\cdot\\
\cdot&\cdot&\cdot&\cdot&\cdot&1&\cdot&\cdot\\
\cdot&\cdot&\cdot&\cdot&1&\cdot&\cdot&\cdot\\
\cdot&\cdot&\cdot&\cdot&\cdot&\cdot&1&\cdot\\
\cdot&\cdot&\cdot&\cdot&\cdot&\cdot&\cdot&1
\end{array}\right)
\]

\caption{The matrices verifying that $W$(E$_8)$ is a strongly real Beauville group.}
\label{e8fig}
\end{figure}

\subsection{The Coxeter group of type F$_4$}

\begin{lemma}
The group $W($\emph{F}$_4$) is not a Beauville group.
\end{lemma}

\begin{proof}
Let $W:=W(F_4)$, let $W'$ be its derived subgroup and let $x$ and
$y$ be their standard generators (that is $x$ is in class 2C,
$o(y)=6$, $o(xy)=6$ and $o(xy^2)=4$). Since $W/W'=2^2$ there are
three index 2 subgroups, namely
\[
H_x:=\langle W',x \rangle\mbox{, }H_y:=\langle W',y\rangle\mbox{ and }H_{xy}:=\langle W',xy\rangle.
\]
Outside $W'$ each of $H_y$ and $H_{xy}$ contain a $W$-class of
elements of order 6 (represented by $y$ and $xy$ respectively) a
class consisting their cubes and for each of these there is a second
class consisting of the same elements multiplied by the central
involution (which is represented by $(xy^3xy)^6$). There is also a
$W$-class of elements of order 4 (represented by $[x,y]y$ and
$[x,y]xy$ respectively).

Consequently, all elements of order 6 in $H_y$ square to the same
class, as do all elements of order 6 in $H_{xy}$. All elements of
$W$ order 4 outside $W'$ power up to elements of class 2B, which is
contained in the derived subgroup. Consequently, if one generating
pair uses an element of order 6 from one of $H_y$ or $H_{xy}$ then
the other generating pair cannot use an element of order 6 from that
same subgroup. Writing $(a,b)-generated$ to mean `the group is
generated by an element of order $a$ from $H_y$ and an element of
order $b$ from $H_{xy}$' it turns out that searching through all
such elements, the group is not (4,4)-generated (4,2)-generated or
(2,4)-generated. Furthermore, the group is neither (6,2)-generated
nor (2,6)-generated. (Though it can be generated by an element of
order 6 from either $H_y$ or $H_{xy}$ along with a class 2C element
from $H_x$ such as the standard generators). Consequently, the group
can only be (6,4)-generated, (4,6)-generated or (6,6)-generated and
no combination of such pairs can form a Beauville structure. We
conclude that $W$ is not a Beauville group.
\end{proof}

\subsection{The Coxeter group of type H$_3$}

\begin{lemma}
The Coxeter group of type \emph{H}$_3$ is not a Beauville group
\end{lemma}

\begin{proof}
Since this group is isomorphic to $2\times Alt(5)$ it is clear that in any generating pair $g, h\in 2\times Alt(5)$ at least one of $g$, $h$ or $gh$ must power-up to the central involution. Consequently, condition ($\dagger$) of Definition \ref{MainDef} cannot be satisfied by any prospective Beauville structure since $\Sigma(x_1,y_1)\cap\Sigma(x_2,y_2)$ must contain the central involution.
\end{proof}

\subsection{The Coxeter group of type H$_4$}

As with the E$_n$ cases we will use an explicit representation of the group that has the helpful property of including all permutation matrices. Unlike those cases, such a representation does not appear in \cite{Edin}. We will use a representation in which the simple reflections correspond to the beautiful matrices given in Figure \ref{h4simples}.

\begin{figure}
\[
\left(\begin{array}{rrrr}
1&\cdot&\cdot&\cdot\\
\cdot&1&\cdot&\cdot\\
\cdot&\cdot&\cdot&1\\
\cdot&\cdot&1&\cdot
\end{array}\right)\mbox{, }
\left(\begin{array}{rrrr}
1&\cdot&\cdot&\cdot\\
\cdot&\cdot&1&\cdot\\
\cdot&1&\cdot&\cdot\\
\cdot&\cdot&\cdot&1
\end{array}\right)\mbox{, }
\left(\begin{array}{rrrr}
\cdot&1&\cdot&\cdot\\
1&\cdot&\cdot&\cdot\\
\cdot&\cdot&1&\cdot\\
\cdot&\cdot&\cdot&1
\end{array}\right)\mbox{ and }
\]
\[
-\frac{1}{4}\left(\begin{array}{rrrr}
1&\sqrt{5}&\sqrt{5}&\sqrt{5}\\
\sqrt{5}&-3&1&1\\
\sqrt{5}&1&-3&1\\
\sqrt{5}&1&1&-3
\end{array}\right).
\]

\caption{The matrices corresponding to the simple reflections in our representation of $W($H$_4)$.}
\label{h4simples}
\end{figure}

\begin{lemma}
The group $W($\emph{H}$_4)$ is a strongly real Beauville group.
\end{lemma}

\begin{proof}
It is a straightforward computation to verify that the matrices given in Figure \ref{h4fig} have the correct properties to be a Beauville structure for this group as well as the fact that $x_i^t=x_i^{-1}$ and $y_i^t=y_i^{-1}$ for $i=1,2$ thereby showing that this is in fact a strongly real Beauville structure.
\end{proof}

\begin{figure}

\[
x_1:=-\frac{1}{4}\left(\begin{array}{rrrr}
-3&1&1&\sqrt{5}\\
1&-3&1&\sqrt{5}\\
1&1&-3&\sqrt{5}\\
\sqrt{5}&\sqrt{5}&\sqrt{5}&1\\
\end{array}\right)
\mbox{, }
y_1:=\left(\begin{array}{rrrr}
\cdot&\cdot&1&\cdot\\
\cdot&\cdot&\cdot&1\\
\cdot&1&\cdot&\cdot\\
1&\cdot&\cdot&\cdot
\end{array}\right)
\]

\[
x_2:=\left(\begin{array}{rrrr}
\cdot&1&\cdot&\cdot\\
\cdot&\cdot&1&\cdot\\
1&\cdot&\cdot&\cdot\\
\cdot&\cdot&\cdot&1
\end{array}\right)\mbox{, }
y_2:=-\frac{1}{4}\left(\begin{array}{cccc}
\cdot&1+\sqrt{5}&-2&-1+\sqrt{5}\\
1+\sqrt{5}&\cdot&-1+\sqrt{5}&2\\
-1+\sqrt{5}&-2&-1-\sqrt{5}&\cdot\\
2&-1+\sqrt{5}&\cdot&-1-\sqrt{5}\\
\end{array}\right)
\]

\[
t:=\left(\begin{array}{rrrr}
\cdot&1&\cdot&\cdot\\
1&\cdot&\cdot&\cdot\\
\cdot&\cdot&1&\cdot\\
\cdot&\cdot&\cdot&1
\end{array}\right)
\]

\caption{The matrices verifying that $W$(H$_4)$ is a strongly real Beauville group.}
\label{h4fig}
\end{figure}

\section{Proofs}\label{proofs}

In this section we pull together the results of the earlier sections to give the overall proof of Theorem \ref{MainThm} and of Corollary \ref{MainCor}. We will also discuss some of the barriers to extending these results.

\begin{proof}
(of Theorem \ref{MainThm}) Part (a) was proved by Fuertes and Gonz\'{a}lez-Diez in \cite[Theorem 3]{FG}. Part (b) was proved in Section \ref{BnSec}; part (c) was proved in Section \ref{DnSec} and part (d) was proved in Section \ref{ExSec}. Finally part (e) is a simple observation first made by Bauer, Catanese and Grunewald in \cite[Lemma 3.7]{BCG1}.
\end{proof}


\begin{proof}
(of Corollary \ref{MainCor}) First note that the abelianisation of a direct product of $n$ irreducible Coxeter groups contains the elementary abelian group of order $2^n$. It follows that the direct product of more than two Coxeter groups is not even 2-generated, let alone a Beauville group, so a Coxeter Beauville group must either be irreducible or of the form $K_1\times K_2$ for some irreducible Coxeter groups $K_1$ and $K_2$.

The groups of type B$_n$ have the property that their abelianisation is the Klein foursgroup, so neither $K_i$ can be of this type.

For the remaining claim note that the Beauville structures constructed in Sections \ref{DnSec} and \ref{ExSec} as well as the Beauville structures previously constructed by the author in \cite[Lemma 19]{NewcProc} in the type A$_n$ case were constructed in such a way that not all of the elements $x_1$, $y_1$, $x_2$ and $y_2$ lie outside the derived subgroup and in particular these structures have the property that $\{\{(x_1,y_2^{-1}),(y_1,x_2^-1)\},\{(x_2,y_1^{-1}),(y_2,x_1^{-1})\}\}$ will generate the whole of $K_1\times K_2$. (Where possible we took the $x_i$s from outside the derived subgroup and the $y_i$s from inside, indeed only in the H$_4$ case did it seems difficult to find such a structure, but even in that case only $x_2$ lies inside the derived subgroup, so the elements of the product still lie in different cosets and thus generate the whole group.) Finally it is easy to see that we can `double up' the automorphisms used earlier to make this structure a strongly real one.
\end{proof}

The general question of classifying precisely which groups of the form $K_1\times K_2$ are strongly real Beauville groups is much more difficult. It is easy to see that neither $K_i$ can be of type F$_4$ nor of type I$_2(k)$, $k$ even, for the same reason that they cannot be of type B$_n$ but the general picture is much more complicated. For example, the below permutations give a strongly real Beauville structure for the group $W(\mbox{H}_3)\times W(\mbox{H}_3)$ (recall that $W(\mbox{H}_3)$ is not a Beauville group)
\[
x_1:=(1,2,3,4,5)(6,7)(8,9)(10,11)\mbox{, }
y_1:=(1,2)(3,4)(8,9,10,11,12)(13,14),
\]
\[
x_2:=(1,2,3)(8,9)(10,11)(13,14)\mbox{, }
y_2:=(1,4)(2,5)(6,7)(8,10,12)(13,14)
\]
whilst the following permutations give a strongly real Beauville structure for the group $W($A$_4)\times W($I$_2(3))$
\[
x_1:=(1,2)(3,4,5)(6,7,8)\mbox{, }
y_1:=(4,1)(2,3)(6,7)
\]
\[
x_2:=(1,2,3,4)(6,7,8)\mbox{, }
y_2:=(2,3,4,5)(6,7).
\]
It is also easy to see that neither $W($A$_4)\times W($I$_2(5))$ nor $W($A$_5)\times W($I$_2(5))$ are Beauville groups since for any generating pair $x$ and $y$ we have that $\Sigma(x,y)$ must contain elements of order 5 that belong solely in the `I$_2$ part' of the group. Similarly if $k$ is coprime to $|K_1|$, then $K_1\times W($I$_2(k))$ cannot be a Beauville group, nor can $W($I$_2(k))\times W($I$_2(k'))$ be a Beauville group for any $k$ and $k'$. This naturally poses the following.

\begin{q}
Which of the groups $K_1\times K_2$ are (strongly real) Beauville groups.
\end{q}

\begin{conj}
A Coxeter group is a Beauville group if and only if it is a strongly real Beauville group.
\end{conj}


We remark that the author had previously proved the special case of the groups $W(\mbox{A}_n)\times W(\mbox{A}_n)$ in \cite[Lemma 19]{NewcProc}.

\section{The mixed case}

In this final section we discuss attempts to use Coxeter groups to construct mixed Beauville groups.

\subsection{Mixed Beauville groups}

Recall from Section \ref{Sec1} that a Beauville group acts on a product of Riemann surfaces $\mathcal{C}\times\mathcal{C}'$. The structures in Definition \ref{MainDef} do not interchange $\mathcal{C}$ with $\mathcal{C}'$. It is interesting to ask when the group can interchange the two curves and if this can be be translated into group theoretic terms. It turns out that the answer is yes and this happens precisely when a group has a so-called mixed Beauville structure that we define as follows \cite{BCG2}.

\begin{definition} \label{mixed}
Let $G$ be a finite group. A \emph{mixed Beauville quadruple} for $G$ is a quadruple $(G^0;a,c;g)$ consisting of a subgroup $G^0$ of index $2$ in $G$; of elements $a,c\in G^0$ and of an element $g\in G$ such that
\begin{itemize}
\item $G^0$ is generated by $a$ and $c$;
\item $g\not\in G^0$;
\item for every $\gamma\in G^0$ we have that $(g\gamma)^2\not\in\Sigma(a,c)$ and
\item $\Sigma(a,c)\cap\Sigma(a^g,c^g)=\{e\}$.
\end{itemize}
If $G$ has a mixed Beauville quadruple, then we say that $G$ is a \emph{mixed Beauville group} and call $(G^0;a,c;g)$ a \emph{mixed Beauville structure} of $G$.
\end{definition}

Mixed Beauville groups are much more difficult to construct than their unmixed counterparts. In slightly different terms, \cite[Lemma 5]{FG} states the
following.

\begin{lemma}\label{FGLem}
Let $G$ and $G^0$ be as in part Definition \ref{mixed}.
Then the order of any element in $G\setminus G^0$ is a multiple of
$4$.
\end{lemma}

Fuertes and G. Gonz\'{a}lez-Diez used the above lemma to show that
the symmetric groups do not contain mixed Beauville structures since
they contain transpositions which of course have order 2.

Here we note that this simple observation immediately carries over
to Coxeter groups in general: if $G$ is of type D$_n$, E$_6$, E$_7$, E$_8$, H$_3$ and H$_4$ then reflections have order 2 and lie outside the only index 2 subgroup and so by Lemma \ref{FGLem} these cannot be mixed Beauville groups.

In the case of the groups of type B$_n$ the abelianisation is the Klein foursgroup so there are three subgroups of index 2. Let $l$ be a
reflection in one of the long root associated with B$_n$ and let $s$
be a reflection in one of the corresponding short roots. It may be
shown that if $W'$ is the derived subgroup then the three subgroups
are $\langle W',l\rangle$, $\langle W',s\rangle$ and $\langle
W',ls\rangle$. In each case there is clearly an element outside the
derived subgroup whose order is not a multiple of 4, so no Coxeter
groups of type B$_n$ of index at most 2 can be made into mixed
Beauville groups. A similar argument shows that the Coxeter group of type F$_4$ is not a mixed Beaville group.

\subsection{Mixable Beauville groups}

In \cite{FP} the author and Pierro introduced the notion of `mixable' Beauville groups that we define as follows.

\begin{definition}\label{MixableDef}
Let $H$ be a finite group. Let $a_1,c_1,a_2,c_2\in H$ and define $\nu(a_i,c_i)=o(a_i)o(c_i)o(a_ic_i)$ for $i=1,2$. If
\begin{enumerate}
\item $o(a_1)$ and $o(c_1)$ are even;
\item $\langle a_1^2,a_1c_1,c_1^2\rangle=H$;
\item $\langle a_2,c_2\rangle=H$ and
\item $\nu(a_1,c_1)$ is coprime to $\nu(a_2,c_2)$,
\end{enumerate}
then we say that $H$ is a \emph{mixable Beauville group}.
\end{definition}

(The above definition is slightly different to that given in \cite{FP} where the definition is restricted to perfect groups. This restriction enables us to replace condition (2) with the slightly simpler condition that $\langle a_1,c_1\rangle=H$. Here we prefer this more general notion as no Coxeter group is perfect since the reflections always lie outside the derived subgroup.)

The reason for introducing these concepts is as follows. If we consider the group $G:=(H\times H):Q_{4k}$ where $Q_{4k}$ is the dicyclic group of order $4k$ and $k>1$ is an integer coprime to gcd$(o(a_1),o(c_1))$, then whenever $H$ is a mixable Beauville group we have that $G$ is a mixed Beauville groups.

\begin{lemma}
No Coxeter group is a mixable Beauville group.
\end{lemma}

\begin{proof}
Any generating set for a group must contain at least one element from outside its derived subgroup. In any Coxeter group, every element outside the derived subgroup has even order, so it is impossible to find generating pairs that satisfy condition (4) of Definition \ref{MixableDef}.
\end{proof}

\textbf{Acknowledgement} The author wishes to express his deepest gratitude to his colleague Professor Sarah Hart and to Professor Christopher Parker for extremely helpful comments on earlier drafts of this work.

\end{document}